\documentclass{amsart}

\author{M. Fuchs, R. Hornung, R. De Bin, A.-L. Boulesteix}
\title{A $U$-statistic estimator for the variance of resampling-based error estimators}
\address{Institut für Medizinische Informationsverarbeitung Biometrie und Epidemiologie, Ludwig-Maximilians-Universität München,\\ Marchioninistr. 15, 81377 München, Germany}
\email{\{fuchs,hornung,debin,boulesteix\}@ibe.med.uni-muenchen.de}

\newtheorem{theorem}{Theorem}[section]

\newtheorem{proposition}[theorem]{Proposition}
\newtheorem{corollary}[theorem]{Corollary}

\theoremstyle{definition}
\newtheorem{definition}[theorem]{Definition}

\newtheorem{assumption}[theorem]{Assumption}

\theoremstyle{remark}
\newtheorem{remark}[theorem]{Remark}

\usepackage{mathptmx}
\usepackage{amsmath}
\usepackage{mathptmx}

 \usepackage{amssymb}
 \usepackage[utf8]{inputenc}
 \usepackage{mathrsfs,nicefrac}
\usepackage{amsxtra}

\usepackage{enumerate}
\usepackage{mathrsfs}
\usepackage{natbib}

\newcommand{\R}{{\mathbb{R}}}

\newcommand{\abs}[1]{\left|#1\right|}
\newcommand{\floor}[1]{\lfloor #1\rfloor}

\newcommand{\oParam}{\Delta}
\newcommand{\U}{{\widehat{\oParam}}}
\newcommand{\Uincompl}{{\widetilde{\oParam}}}
\newcommand{\defequal}{:=}

\newcommand{\E}[2]{{\mathbb{E}}_{#1}(#2)}

\newcommand{\V}[2]{{\mathbb{V}}_{#1}(#2)}

\newcommand{\Cov}[2]{{\textrm{cov}}_{#1}(#2)}

\newcommand{\z}{Z}

\newcommand{\VarEst}{\hat{v}}
\newcommand{\wVU}{\widehat{w}}

\begin{document}


\begin{abstract}
We revisit resampling procedures for error estimation in binary classification in terms of $U$-statistics.
In particular, we exploit the fact that the error rate estimator involving all learning-testing splits is a $U$-statistic.
Therefore, several standard theorems on properties of $U$-statistics apply.
In particular, it has minimal variance among all unbiased estimators and is asymptotically normally distributed.
Moreover, there is an unbiased estimator for this minimal variance if the total sample size is at least the double learning set size plus two.
In this case, we exhibit such an estimator which is another $U$-statistic.
It enjoys, again, various optimality properties and yields an asymptotically exact hypothesis test of the equality of error rates when two learning algorithms are compared.
Our statements apply to any deterministic learning algorithms under weak non-degeneracy assumptions.
In an application to tuning parameter choice in lasso regression on a gene expression data set, the test does not reject the null hypothesis of equal rates between two different parameters.
\end{abstract}
\maketitle

\section{Introduction}\label{Intro}
The goal of supervised statistical learning is to develop prediction rules taking the values of predictor variables as input and returning a predicted value of the response variable. A prediction rule is typically learnt by applying a learning algorithm $M$ to a so-called learning data set. A typical example in biomedical research is the prediction of patient outcome (e.g.\ recidive/no recidive within five years, tumor class, lymph node status, response to chemotherapy, etc.) based on bio-markers such as, e.g., gene expression data. The practitioners are usually interested in the accuracy of the prediction rule learnt from their data set to predict future patients, while methodological researchers rather want to know whether the learning algorithm is good at learning accurate prediction rules for different data sets drawn from a distribution of interest. The first perspective is called ``conditional'' (since referring to a specific data set) while the latter, which we take in this paper, is denoted as ``unconditional''. Precisely, this paper focuses on the parameter defined as the difference between the unconditional errors of two learning algorithms of interest, $M$ and $M'$, for binary classification.

If the data set is very large, one can observe independent realizations of estimators of the unconditional error rates and use them for a paired $t$-test (see Section~\ref{tte}). In practise, however, huge data sets are rarely available. Prediction errors are thus usually estimated by resampling procedures consisting of splitting the available data set into learning and test sets a large number of times and averaging the estimated error over these iterations. The well-known cross-validation procedure can be seen as a special case of resampling procedure for error estimation. A detailed overview of the vast literature on cross-validation would go beyond the scope of this paper. The reader is referred to~\citet{Arlot:2010} for a comprehensive survey.

Having estimated the error rate, it is typically of interest to test the null hypothesis of equal error rates between learning algorithms. This requires insight into the estimator's variance. Resampling-based error estimators typically have a very large variance, in particular in the case of small samples or high-dimensional predictor space \citep{Dougherty:2011}. The estimation of this variance has been the focus of a large body of literature, especially in the machine learning context. A good estimation of the variance of resampling-based error estimators would allow to, e.g., derive reliable confidence intervals for the true error or to construct statistical tests to compare the performance of several learning algorithms. The latter task is of crucial importance in practise, since applied computational scientists including biostatisticians often have to make a choice within a multitude of different learning algorithms whose performance in the considered settings is poorly explored. In \citet{Wiel:2009}, for each splitting in repeated sub-sampling the predictions of the two classifiers are compared by a Wilcoxon-test, and the resulting $p$-values are combined. In \citet{Jiang:2008}, the authors show the asymptotic normality of the error rate estimator in the case of a support vector machine. In \citet{Jiang:2008:2}, a bias-corrected bootstrap-estimator for the error rate from leave-one-out cross-validation is introduced.

Various estimators of the variance of resampling-based error estimators have been suggested in the literature \citep{Dietterich:1998,Nadeau:2003}, most of them based on critical simplifying assumptions. As far as cross-validation error estimates are concerned, \citet{Bengio:2003} show that there exists no unbiased estimator of their variance. To date, the estimation of the variance of resampling-based error rates in general remains a challenging issue with no adequate answer yet both from a theoretical and practical point of view. In particular, there are no exact nor even asymptotically exact test procedures for testing equality of error rates between learning algorithms available. The present paper shows that there is an asymptotically exact test for the comparison of learing algorithms by using and extending results from $U$-statistics theory.

Despite the large body of literature, there seems to be no explicit treatment of the asymptotic properties of the estimators in general. Our main results are Theorem~\ref{MainTheorem}, stating that there is an unbiased estimator of the difference estimator's variance if $n\ge 2g + 2$, where $g$ is the learning set size and $n$ is the sample size, and Theorem~\ref{AA}, providing the central limit theorem for the studentized statistic as it is needed for testing. The use of only half the sample size for learning has already occurred in the literature, in a roughly similar context and on grounds of intuitive reasoning \citep[Section 10.2.1]{Buehlmann:2011}. 

Corollary~\ref{fast} gives an explicit bound on the number of iterations necessary to approximate the leave-$p$-out estimator, i.e. the minimum variance estimator, to an arbitrary given precision, where $p = n-g$; we show that this minimal variance can be estimated by an unbiased estimator, namely that from Definition~\ref{MainDefi}. It has minimal variance itself, and the ensuing studentized test in~\eqref{test} is asymptotically exact. This shows that it is not necessary to endeavour in determining the distribution of combinations of $p$-values to test equality of error rates, as in \citet[Section 2.3]{Wiel:2009}.

Section~\ref{Defs} recalls important definitions pertaining to $U$-statistics and cross-validation viewed as incomplete $U$-statistics. We show that the procedure which involves all learning-testing splits is then a complete $U$-statistic. In Section~\ref{VarEstimator}, we show that $U$-statistics theory naturally suggests an unbiased estimator of the variance of this estimated difference of errors as soon as the sample size criterion is satisfied.
In Section~\ref{Testing}, we exploit this variance estimator to derive an asymptotically exact hypothesis test of equality of the true errors of two learning algorithms $M$ and $M'$. Section~\ref{Practice} addresses numerical computation of approximations of the leave-$p$-out cross-validation estimator, while an illustration of the variance estimation and the hypothesis test through application to the choice of the penalty parameter in lasso regression is presented in Section~\ref{Results}.
\normalfont

\section{Definitions, notations and preliminaries}\label{Defs}
\subsection{Hoeffding's definition}\label{HoeffdingDef}
Since the complete error estimator that we will consider from the next section on is a $U$-statistic, we start by recalling the definition of $U$-statistics and their basic properties. In the following, the reader who is already familiar with the machine learning context may take $\Psi_0 \defequal \Phi_0$ and $m \defequal g+1$ at first, where $g$ is the learning sample size and $\Phi_0$ is as defined in~\eqref{PhiformulaSymmetrized}; however, this is not necessary and we will need other cases of the definition as well.
\begin{definition}[$U$-statistic, \cite{Hoeffding:1948}]\label{U}
  Let $({\bf \z}_i)$, $i = 1,\dotsc,n$ be independent and identically distributed $r$-dimensional random vectors with arbitrary distribution. Let $m\le n$, and let $\Psi_0: \R^{r\times m}\to \R$ be an arbitrary measurable symmetric function of $m$ vector arguments. Write $\Psi_0(S)$ as above for the well-defined value of $\Psi_0$ at those $Z_i$ with indices from $S \subset \{1, \dotsc, n\}, \abs{S} = m$. Consider the average of $\Psi_0$ in the maximal design $\mathscr{S}$ of unordered size-$m$ subsets
\begin{equation}\label{Uformula}
U = U({\bf \z}_1,\dotsc, {\bf \z}_n)= \binom{n}{m}^{-1} \sum_{S \in \mathscr{S}} \Psi_0(S).
\end{equation}
Any statistic of such a form is called a $U$-statistic.
\end{definition}
The trailing factor is the inverse of the number of summands, the cardinality $\abs{\mathscr{S}}$. So, a $U$-statistic is an unbiased estimator for the associated parameter
\begin{equation}\label{rp}
\Theta(P) = \idotsint \Psi_0({\bf z}_1, \dotsc, {\bf z}_m) dP({\bf z}_1) \dotsm dP({\bf z}_m)
\end{equation}
for a probability distribution $P$ on $\R^r$, where ${\bf z}_1, \dotsc, {\bf z}_m$ are $r$-vectors. A parameter of this form is called a regular parameter. If $m$ is the smallest number such that there exists such a symmetric function $\Psi_0$ that represents a given parameter $\Theta(P)$ in the form~\eqref{rp}, then $m$ is called the degree of $\Psi_0$ or of $\Theta(P)$, and the function $\Psi_0$ is called a kernel of $U$. Furthermore, $U$ has minimal variance among all unbiased estimators over any family of distributions $P$ containing all properly discontinuous distributions. \citet{Hoeffding:1948} shows the asymptotic normality of $U$-statistics as $n\to\infty$ and that a good number of well-known statistics such as sample mean, empirical moments, Gini coefficient, etc.\ are subsumed by the definition.

It is also possible to associate a $U$-statistic to a non-symmetric kernel $\Psi$, i.e.\ to estimate $\E{}{\Psi}$ by a $U$-statistic. The cost is to deal with $n!/(n-m)!$ summands, many more than just the binomial coefficient $n!/(m!(n-m)!)$. The symmetrization, indeed, consists in grouping all $m!$ summands involving the same unordered index set together. This writes a $U$-statistic with non-symmetric kernel $\Psi$ in the form of \citet[Section 4.4]{Hoeffding:1948} with symmetric kernel $\Psi_0$ as in \citet[Section 3.3]{Hoeffding:1948}.
\subsection{The difference of true error rates as the parameter of interest}
The goal of this section is to formalize the true error rate and to recall its nature as an expectation. Let $\mathcal{X}=\R^{r-1}, r\in \mathbb{N},$ be a fixed predictor space, and $\mathcal{Y}\subset \R$ be a space of responses. The number $r$ will be as in~\citet{Hoeffding:1948} and one would usually denote $p = r-1$. Assume there is an unknown but fixed probability distribution $P$ on $\mathcal{X}\times \mathcal{Y}\subset \R^r$ defined on the $\sigma$-algebra of Lebesgue-measurable sets. We do not require $P$ to be absolutely continuous with respect to the Lebesgue measure in order to allow for a discrete marginal distribution in $\mathcal{Y}$ as it occurs in binary classification. The distribution $P$ can be thought of as being supported on $\R^r$ instead of $\mathcal{X}\times \mathcal{Y} \subset \R^r$ only, by identifying $P$ with its push-forward image $i_*(P)$ under the inclusion $i: \mathcal{X}\times \mathcal{Y} \to \R^r$. This allows to apply \cite{Hoeffding:1948} which only describes $U$-statistics on a Euclidean space $\R^r$ to the definition and investigation of $U$-statistics on (products of) $\mathcal{X}\times \mathcal{Y}$. Let us fix a loss function $L:\mathcal{Y}\times \mathcal{Y}\to \R$. Typically, $L$ is the misclassification loss $L(y_1, y_2) = 1_{y_1 \ne y_2}$, but can be an arbitrary measurable function. Since we suppose the marginal distribution of $P$ on $\mathcal{Y}$ to be discrete, the loss function associated as done below to a learning algorithm is almost surely bounded. Therefore, all moments exist, which will be helpful throughout the paper. It is automatic for the misclassification loss. However, some of the following also work for an unbounded distribution on $\mathcal{Y}$. In that case one would typically consider losses which are not almost surely bounded, such as, for instance, the residual sum of squares loss $L(y_1, y_2) = (y_1 - y_2)^2$ or the Brier score in survival analysis. We will not work out the details of unbounded losses.

Say we are interested in the difference of error rates of classifiers learnt on a sample of size $g$. Typical choices are $g = 4n/5$ (assuming that five divides $n$) for a learning/testing sample size ratio of $4:1$ and $g = n - 1$ for leave-one-out cross-validation. We also allow for $g = 0$ in case we are interested in the performance of classification rules that were already learnt on different and fixed data. In that case, it is important that the learning data were different since otherwise there would be a problematic contradiction between simultaneously regarding the data as fixed and as being drawn from $P$. 

Let $({\bf z}_i)_{i = 1,\dotsc ,n} = ({\bf x}_1, y_1, \dotsc, {\bf x}_n, y_n)$, where ${\bf x}_i \in \mathcal{X}$ and $y_i \in \mathcal{Y}$, and denote by
\[
f_{M}:(\mathcal{X}\times \mathcal{Y})^{\times g} \times \mathcal{X} \to \mathcal{Y}
\]
the function that maps $({\bf z}_1, \dotsc, {\bf z}_g; {\bf x}_{g + 1}) \in (\mathcal{X}\times \mathcal{Y})^{\times g} \times \mathcal{X}$ to the prediction, an element of $\mathcal{Y}$, by the learning algorithm $M$, learnt on ${\bf z}_1, \dotsc, {\bf z}_g$ and applied to ${\bf x}_{g + 1}$. We are only concerned with deterministic learning algorithms $M$, i.e.\ ones which do not involve any random component for classification. We suppose that $f_M$ is symmetric in the first $g$ entries ${\bf z}_1, \dotsc, {\bf z}_g$, i.e.\ $M$ treats all learning observations equally, and that $f_M$ is measurable with respect to the product $\sigma$-algebra. The inclusion $\mathcal{X} \times \mathcal{Y} \to \R^r$ defines an inclusion $(\mathcal{X}\times \mathcal{Y})^{\times g} \times \mathcal{X} \to \R^{rg + r - 1}$ and in order to be able to apply \cite{Hoeffding:1948} we view $f_M$ as a map on $\R^{rg + r - 1}$ by extending it by zero on $\R^{rg + r - 1} \backslash (\mathcal{X}\times \mathcal{Y})^{\times g} \times \mathcal{X}$ which is a null-set with respect to the push-forward measure $i_*(P)$. The map $f_M$ is then also measurable on $\R^{rg + r - 1}$.

Denote by $\Phi:(\mathcal{X}\times \mathcal{Y})^{g + 1}\to \R$ the function
\begin{equation}\label{Phiformula}
\Phi({\bf z}_1, \dotsc, {\bf z}_g; {\bf z}_{g + 1}) \defequal
    L(f_M({\bf z}_1, \dotsc, {\bf z}_g; {\bf x}_{g + 1}), y_{g + 1}) - 
    L(f_{M'}({\bf z}_1, \dotsc, {\bf z}_g; {\bf x}_{g + 1}), y_{g + 1})
\end{equation}
for two learning algorithms $M$ and $M'$.
The value $\Phi({\bf z}_1, \dotsc, {\bf z}_g; {\bf z}_{g + 1})$ is the empirical difference of error rates between $M$ and $M'$, learnt on the first $g$ observations $({\bf z}_1, \dotsc, {\bf z}_g)$ of the sample $({\bf z}_i)$ and evaluated on the single last entry. The semicolon thus visually separates learning and test sets. The definition of $\Phi$ involves only a single test observation. Anticipating a little, the reason is that $\Phi$ had to be defined with a minimal number of arguments necessary for~\eqref{Exp} below. In the following, we will conveniently consider larger test sample sizes by applying the mechanism of associating a $U$-statistic to a kernel.

As noted above, we assume that $\Phi$ is almost surely bounded. This happens, for instance, for bounded loss functions such as the misclassification loss. Also, we may view $\Phi$ as being defined on $\R^{r(g+1)}$ instead of $(\mathcal{X}\times \mathcal{Y})^{g+1}$ without notational distinction, in the same way as $f_M$ was extended to $\R^{rg + r - 1}$.

The true difference of error rates between $M$ and $M'$ is the expectation of $\Phi$, taken with respect to $g+1$ independent realizations of $P$:
\begin{equation}\label{ThetaIsRegPa}
\begin{split}
\oParam & \defequal e_M - e_{M'} = \E{
P^{\otimes {(g + 1)}}}{\Phi({\bf \z}_1, \dotsc, {\bf \z}_g; {\bf \z}_{g+1})}\\
&= \idotsint_{(\mathcal{X}\times \mathcal{Y})^{\times (g +1)}}  \bigl(L(f_{M}({\bf z}_1, \dotsc, {\bf z}_g; {\bf x}_{g + 1}), y_{g + 1})\\
& - L(f_{M'}({\bf z}_1, \dotsc, {\bf z}_g; {\bf x}_{g + 1}), y_{g + 1}) \bigr) \thinspace dP({\bf z}_1) \dotsm dP({\bf z}_{g+1}),
\end{split}
\end{equation}
where both learning and test data are random. The existence of the expectation follows from measurability and from the boundedness assumption. The quantity $\oParam$ is the parameter of main interest. Also, we consider the symmetric function of $g+1$ arguments
\begin{equation}\label{PhiformulaSymmetrized}
\Phi_0({\bf z}_1, \dotsc, {\bf z}_{g + 1}) \defequal \frac{1}{g+1}\sum_{i = 1}^ {g+1}\Phi({\bf z}_1, \dotsc, {\bf z}_{i-1}, {\bf z}_{i+1}, \dotsc, {\bf z}_{g+1}; {\bf z}_i),
\end{equation}
satisfying
\begin{equation}\label{Exp}
\oParam = \E{}{\Phi} = \E{}{\Phi_0}.
\end{equation}
For the particular non-symmetric kernel $\Psi = \Phi$ introduced in~\eqref{Phiformula}, the symmetrization $\Psi_0 = \Phi_0$ of~\eqref{PhiformulaSymmetrized} can be written involving only $m = g+1$ summands instead of $m!$, in other words only cyclic permutations instead of all permutations, due to the assumption that learning is symmetric. In practise, it is not advantageous to compute $\Phi_0$ directly because a learning procedure should be used on more than just one test observation for numerical efficiency (see Section~\ref{Practice}); however, it is very convenient to consider $\Phi_0$ for ease of presentation.



\begin{remark} In case one is interested in estimating $e_M$ only instead of a difference $\oParam = e_M - e_{M'}$, one can set the second summand of~\eqref{Phiformula} identically to zero. We will not go into the details.
\end{remark}
\subsection{Tests of the true error rate}\label{tte}
In this section, let us recall the test problem of interest. We want to test the null hypothesis $H^0: \E{}{\Phi} = 0$ against the alternative $H^1: \E{}{\Phi}\ne 0$. The former is usually called the unconditional null hypothesis \citep{Braga-Neto:2004}.
\begin{remark}
  There is also a conditional null hypothesis where the classification rule is supposed to be given, for instance learnt on fixed independent data, and the expectation is taken only with respect to the test set. However, the learning data are usually also random and may even be modelled to be from $P$ as well. In this case, the conditional null depends on random data, leading to severe difficulties in the interpretation of type one error. For this reason, in this paper we will only consider the unconditional error rate. However, setting $g = 0$ and plugging in a ready-made classification rule for $\Phi$, regardless of the data it was learnt on, leads to a sort of conditional null hypothesis. In this case, the true error becomes a random variable of the learning data, and the latter \emph{must not} be from the sample $(z_1, \dotsc, z_n)$. We will not go into details of conditional testing or of the case $g = 0$.
\end{remark}

The form of $H^ 0$ suggests a $t$-test. However, the number of independent realizations of $\Phi({\bf z}_1, \dotsc, {\bf z}_{g + 1})$ is only $\floor{n/(g + 1)}$, since it is to be computed with respect to $P^{\otimes (g+1)}$. Therefore, a correct $t$-test would be severely underpowered, and cross-validation procedures are usually preferred.
\subsection{Cross-validation}\label{CV}
Let us now show how cross-validation procedures fit into the framework described above. In a cross-validation procedure, dependent realizations of $\Phi({\bf z}_1, \dotsc, {\bf z}_{g + 1})$ are considered. More precisely, for every ordered subset $T = (i_1, \dotsc, i_g; i_{g+1})$ of $\{1, \dotsc, n\}$,
\begin{equation}\label{uincompl}
\Uincompl(T) \defequal \Phi({\bf z}_{i_1}, \dotsc, {\bf z}_{i_g}; {\bf z}_{i_{g+1}})
\end{equation}
is an estimator of $\oParam$, where we visually separate learning and test sets again. We view $\Uincompl(T)$ as an estimator of the difference of error rates of classifiers learnt on samples of size $g$ instead of $n$, in contrast to differing usage in the literature. Thus, $\Uincompl(T)$ is unbiased. Of course, the word ``ordered subset'' refers to the order $1<\dots< g+1$; it is not imposed that $i_1<\dots < i_{g+1}$. Similarly, if $S =\{i_1, \dotsc, i_{g+1}\}$ is an \emph{unordered} subset of size $g+1$ of $\{1, \dotsc, n\}$, the value $\Phi_0({\bf z}_{i_1}, \dotsc, {\bf z}_{i_{g+1}})$, using the symmetric $\Phi_0$ instead of $\Phi$, does not depend on the order of $S$. Therefore, we can unambiguously extend definition~\eqref{uincompl} to a function $\Uincompl$ also on the collection of unordered subsets $S$ by setting $\Uincompl(S)\defequal\Phi_o(S)$. This is an unbiased estimator of $\oParam$. Also, let $\mathscr{T}$ be a collection of ordered subsets $T$ as above. Then, let
\begin{equation}\label{uincompl2}
\Uincompl(\mathscr{T}) \defequal \frac{1}{\abs{\mathscr{T}}} \sum_{T\in \mathscr{T}} \Uincompl(T)
\end{equation}
be the average of all values of $\Uincompl(T)$ over $\mathscr{T}$, and similarly $\Uincompl(\mathscr{S})$ for a collection $\mathscr{S}$ of unordered subsets $S$ as above the average of all values $\Uincompl(S)$ involving the symmetric function $\Phi_0$. Equation \eqref{uincompl2} may involve each learning set multiple times because each observation of a test sample can then contribute a summand to \eqref{uincompl2}. In other contexts the mean error rate over the entire test sample is counted as only one occurrence of the learning set.

For any such collections $\mathscr{T}$ or $\mathscr{S}$, the estimators $\Uincompl(T)$ and $\Uincompl(S)$ are unbiased for $\oParam$. As soon as $\mathscr{T}$ contains together with an ordered subset $T = (i_1, \dotsc, i_g; i_{g+1})$ all its cyclic permutations $(i_2, \dotsc, i_{g+1}; i_1), (i_3, \dotsc, i_{g+1}, i_1; i_2)$ and so on, we have $\Uincompl(\mathscr{T}) = \Uincompl(\mathscr{S})$ where the collection $\mathscr{S}$ is obtained from the collection $\mathscr{T}$ by forgetting the order (and the multiple entries with the same order coming from the cyclic permutation).

Now, the ordinary $K$-fold cross-validation can be incorporated in this framework as follows. Suppose that $K$ is such that $K(n-g) = n$, possibly after disregarding a few observations in order to assure divisibility of $n$ by $n-g$. Therefore, $g \ge n/2$. The extreme cases are $g = n/2$ for $K=2$ and $g = n-1$ for $K = n$. Let $\mathscr{T}_{\text{\emph{CV}}}$ be a collection of ordered subsets of the form
\begin{equation}\label{cvT}
T = \bigl(1, \dotsc, {k(n-g)}, {(k+1)(n-g) + 1}, \dotsc, n; t\bigl)
\end{equation}
where $k = 0, \dotsc, K -1$ enumerates the learning blocks, the notation is to be read in such a way that if $k=0$ the first entry is ${n - g + 1}$ and if $k = K-1$ the last one is $n$, and $t\in \{k(n-g) + 1, \dotsc,  (k+1)(n-g)\}$ enumerates all test observations distinct from the learning block. Thus, $T$ consists of one or two learning strides whose indices are contiguous and whose sizes add up to $g$, together with a single test observation index distinct from any learning observation index. Then $\Uincompl(\mathscr{T}_{\text{\emph{CV}}})$ recovers the ordinary cross-validation estimator of $\oParam$. In practise, one may also compute $\Uincompl(\mathscr{T}_{\text{\emph{CV}}})$ from a permutation of the data, but this does not influence the formal description because $P^{\otimes n}$ is permutation-invariant.
\begin{definition}
We will in general refer to estimators of the form $\Uincompl(\mathscr{T})$ or $\Uincompl(\mathscr{S})$ given by~\eqref{uincompl2}, as to \emph{cross-validation-like procedures}, irrespectively of the structure of $\mathscr{T}$ or $\mathscr{S}$.
\end{definition}
It was shown in \citet{Bengio:2003} that there is no unbiased estimator of the variance $\V{P^{\otimes n}}{\Uincompl(\mathscr{T}_{\text{\emph{CV}}})}$ for any cross-validation procedure $\mathscr{T}_{\text{\emph{CV}}}$, i.e.\ any divisor $K$ of $n$.

It seems plausible from this tedious description of cross-validation that such a particular design $\mathscr{T}_{\text{\emph{CV}}}$ consisting merely of sets of the special form~\eqref{cvT} does not lead to a globally small variance of $\Uincompl(\mathscr{T}_{\text{\emph{CV}}})$ among all possible designs $\mathscr{T}$ with fixed learning set size $g$. This variance is minimal for the cross-validation-like procedure $\mathscr{T}_{\text{\emph{max}}}$ consisting of \emph{all} size $(g+1)$-subsets. We will expose the cases where there is an unbiased variance estimator of it, in contrast to the cross-validation case. Let us call this $\mathscr{T}_{\text{\emph{max}}}$ the maximal design. Another immediate advantage of it over an incomplete one is the fact that the need for a balanced data set, i.e. algorithms whose class labels are equally frequent, and/or for balanced blocks falls away. The only case of $g$ and $K$ such that $\mathscr{T}_{\text{\emph{CV}}} = \mathscr{T}_{\text{\emph{max}}}$ is the leave-one-out case $g = n-1$ respectively $K = n$.

Similarly, one can distinguish those cases of $g$ respectively $K$ such that the associated design $\mathscr{T}$ contains along with an ordered subset $T$ all its cyclic permutations. In such a case, $\Uincompl(\mathscr{T}) = \Uincompl(\mathscr{S})$ for the design $\mathscr{S}$ corresponding to $\mathscr{T}$. Among the cross-validation procedures, only the leave-one-out case $g = n -1$ respectively $K = n$ produces this situation. However, among the cross-validation-like procedures, this can happen for any $g$. For instance, it holds for the maximal design for all $0\le g \le n-1$. This is important to keep in mind for numerical implementation.
\subsection{The full cross-validation-like estimator of $\oParam$ is a $U$-statistic}
In this section, we show that the cross-validation-like procedure with maximal design, where all size-$g$-subsets of the sample are used for learning, is a $U$-statistic and therefore has least variance among all cross-validation-like procedures. It seems that this fact has not yet been described in the literature. Among the immediate consequences of interpreting this procedure as a $U$-statistic will be asymptotic normality, the first case of Theorem~\ref{AA}. The parameter of interest $\Theta = \oParam$ is a regular parameter because of~\eqref{ThetaIsRegPa}; this equation also shows that its degree is at most $g + 1$.
\begin{assumption}
\label{assumptionZ}
The degree of $\oParam$ is exactly $g + 1$.
\end{assumption}
This states that the true error rate cannot be computed from learning samples of smaller size than $g$ for all (reasonable) distributions $P$. While it is not automatic, it seems to be violated only in irrelevant artificial counterexamples, such as for instance one of the form $\Phi(z^1, z^2, z^3) = \Phi'(z^2, z^3)$ where the learning step only makes use of a part of the learning set observations, and in similar cases. So, the assumption is natural.

Let $\mathscr{S}_{\text{\emph{max}}}$ and $\mathscr{T}_{\text{\emph{max}}}$ be the maximal designs of unordered and ordered subsets, respectively, as introduced above. The corresponding error rate estimator is then the $U$-statistic associated to the particular kernel $\Psi = \Phi$ and $\Psi_0 = \Phi_0$, respectively. We define
\begin{equation}\label{ourUForDeltae}
\U \defequal U(\Phi_0) = \Phi_0(\mathscr{S}_{\text{\emph{max}}}) = \Phi(\mathscr{T}_{\text{\emph{max}}})
\end{equation}
as the associated $U$-statistic as in Definition~\ref{U}, i.e.\ the one defined by the symmetric kernel $\Phi_0$. It follows immediately from \citet{Hoeffding:1948} that it has minimal variance among all unbiased estimators of $\oParam$. In particular, it has strictly smaller variance than all cross-validation procedures for $2\le g \le n-2$, and is equal to the cross-validation estimator in the leave-one-out case $g = n-1$. \citet[Section 4.3, Theorems 1 and 4]{Lee:1990} describes quantitatively the variance decrease of $\U$ with respect to $\Uincompl$. These theorems treat the case of a fixed $\mathscr{S}$ and an $\mathscr{S}$ consisting of random subsets, respectively. The statistic $\U$ coincides with what is called complete cross-validation in \citet{Kohavi:1995}, as well as with complete repeated sub-sampling as considered in \cite{Boulesteix:2008}, or leave-$p$-out cross-validation in \cite{Shao:1993} and \cite{Arlot:2010}, where $p = n-g$.

In practise, the definition of $\U$ involves too many summands for computation, but can be easily approximated to arbitrary precision using an $\mathscr{S}$ of random subsets, see Section~\ref{Practice}.
\section{A $U$-statistic estimator of $\V{}{\U}$}\label{VarEstimator}
\subsection{Variances are regular parameters}
The theory of $U$-statistics comes to full power as soon as not only the original regular parameter $\oParam$ is estimated optimally by a $U$-statistic $\U$, but also the variance $\V{}{\U}$ of this $U$-statistic itself is exhibited as another regular parameter, this time depending not only on $\Phi_0$ but also on $n$. Therefore, we are in a position to estimate $\V{}{\U}$ by a $U$-statistic as well.

In the following Proposition, we outline formally that variances and covariances are regular parameters in general, without determining optimally the degree. Thus, the full power of $U$-statistics can be used to estimate them. We then pin down the degree in Proposition~\ref{VarIsU}.
\begin{proposition}\label{UVar}
Let $f({\bf z}_1, \dotsc, {\bf z}_k)$ be a function of $k$ realizations of independent identically distributed random variables ${\bf \z}_i\sim P$ with existing variance $\V{P^{\otimes k}}{f} < \infty$. Then the variance $\V{P^{\otimes k}}{f}$ is a regular parameter of degree at most $2k$. More generally, the covariance between two such functions $f$ and $g$, as soon as it exists, is a regular parameter of degree at most $2k$. 
\end{proposition}
\begin{proof}
Both $\V{}{f}$ and $\Cov{}{f, g}$ are, by definition, polynomials of integrals with respect to $P$. In order to show that they are regular parameters, we have to rewrite each one as a single integral instead. This is accomplished by
\begin{equation}\label{cc}
\begin{split}
\V{P^{\otimes k}}{f} &= \E{}{f^2} - \E{}{f}^2\\
&= \idotsint \frac{1}{2}\bigl(f({\bf z}_1, \dotsc, {\bf z}_k) - f({\bf z}_{k + 1}, \dotsc, {\bf z}_{2 k})\bigr)^2 dP({\bf z}_1) \dotsm dP({\bf z}_{2k})
\end{split} 
\end{equation}
and an almost analogous formula for the covariance $\Cov{P^{\otimes k}}{f, g}$.
\end{proof}
The integrand is not unique. It was chosen in such a way to resemble the symmetric kernel $(z_1 - z_2)^2/2$ of the variance of $P$ itself, i.e.\ the case $r = 1, f(z_1) = z_1$. Furthermore, the degree of $\V{}{f}$, i.e.\ the minimal length of an integrand that accomplishes this, can be much smaller than $2k$ and depends on $f$. Also, the integrand of~\eqref{cc} is not symmetric in general and remains to be symmetrized.

Let us now investigate the case where $f$ is a $U$-statistic associated to a symmetric kernel $\Phi_0$. Caution has to be taken because the regular parameter now depends on $n$, in sharp contrast to the $U$-statistic $\U$ itself. For the case $f = \U$, we have $k = n$, so our knowledge attained so far on the degree of the kernel of $\V{}{\U}$ is that it is at most $2n$. However, it is possible to obtain better insight into the degree of the variance. It will turn out that the variance is a linear combination of regular parameters, each of whose degrees do \emph{not} depend on $n$, only the coefficients of the linear combination depend on $n$. This is the content of the following proposition, which presents in short form results of \citet[Section 5]{Hoeffding:1948} as well as immediate consequences.

In the following, we will consider a general underlying $U$-statistic $U$ which estimates an unknown parameter $\Theta$, and develop the theory of its variance as it is needed for its estimation. From Section~\ref{UForVar} on, we will pay particular attention to the case where $U$ is associated to the kernel $\Phi_0$ defined by~\eqref{PhiformulaSymmetrized}, thus $\Theta = \oParam$ and $U = \widehat{\oParam}$.
\begin{proposition}\label{VarIsU}
Let $U$ be the $U$-statistic associated to a bounded symmetric kernel $\Phi_0$ of degree $m$ and to a total sample size $n$. Denote $\Theta =\E{}{\Phi_0}$. Then the variance of $U$ is a regular parameter of degree at most $2m$. Furthermore, it splits as a sum
\begin{equation}\label{lc}
\V{}{U} = \sum_{c = 1}^m \alpha_c\kappa_c - (1 - \alpha_0)\Theta^2,
\end{equation}
where $\alpha_c$ is the mass function at $c$ of the hyper-geometric distribution $\mathcal{H}(n, m, m)$, and all $\kappa_c$ are regular parameters satisfying
\begin{equation}\label{KappaIsRegPa}
\kappa_c =\idotsint \Phi_0({\bf z}_1, \dotsc, {\bf z}_m)\Phi_0({\bf z}_{m-c+1},\dotsc, {\bf z}_{2m-c}) dP({\bf z}_1)\dotsm dP({\bf z}_{2m-c}).
\end{equation}
Thus, $\kappa_c$ is a regular parameter of degree at most $2m-c$.
Furthermore, since
\begin{equation}\label{ThetaSqIsRegPa}
\Theta^2 = \idotsint \Phi_0({\bf z}_1, \dotsc, {\bf z}_m)\Phi_0({\bf z}_{m+1},\dotsc, {\bf z}_{2m}) dP({\bf z}_1)\dotsm dP({\bf z}_{2m}),
\end{equation}
$\Theta^ 2$ is a regular parameter of degree at most $2m$.
\end{proposition}
\begin{proof}
Direct computation shows that the right hand side of~\eqref{KappaIsRegPa} coincides with what is called $\E{}{\Phi_c^2(X_1, \dotsc, X_c)}$ in \citet[Section 5]{Hoeffding:1948} for all $1\le c \le m$. This step involves the symmetry of the kernel $\Phi_0$ and careful renaming of the variables. Hoeffding already supposes a symmetric kernel which he calls $\Phi$. The quantities $\zeta_c$ of~\citet{Hoeffding:1948} -- which are called $\sigma_c$ in \citet{Lee:1990} -- are thus related to our $\kappa_c$ by means of the equation $\zeta_c = \kappa_c - \Theta^2$, as follows from \citet[formula 5.10]{Hoeffding:1948}. From $\V{}{U} = \sum_{c=1}^m\alpha_c\zeta_c$ \citep[5.13]{Hoeffding:1948} we thus deduce $\V{}{U} = \sum_{c = 1}^m\alpha_c(\kappa_c - \Theta^2) = \sum_{c=1}^m\alpha_c\kappa_c - (1-\alpha_0)\Theta^2$ because $\sum_{c=1}^m\alpha_c
= (1- \alpha_0)$.

The fact that linear combinations of regular parameters are regular parameters \citep[Page 295]{Hoeffding:1948} completes the proof.
\end{proof}
Proposition~\ref{VarIsU} achieves the desired simplification: The degree of $\V{}{U}$ for a $U$-statistic $U$ of degree $m$ is shown to be at most $2m$ instead of $2n$, and the dependence of $\V{}{U}$ on $n$ is now expressed solely by means of the hyper-geometric mass function, whereas $\kappa_c$ and $\E{}{U}^2$ do not depend on $n$.
\begin{remark}\label{kioz}
Direct computation yields $\E{}{U^2} = \sum_{c=1}^m\alpha_c\kappa_c + \alpha_0\Theta^2$, making use of the fact that the kernel is symmetrized. This together with the usual decomposition $\V{}{U} = \E{}{U^2} - \E{}{U}^2 = \E{}{U^2} - \Theta^2$ also proves~\eqref{lc} and shows that the degree of $\E{}{U^2}$ is at most $2m$. It is natural to assume that its degree is exactly $2m$, in analogy to assumptions~\ref{assumptionZ} above and~\ref{assumptionA} below. In contrast, the advantage of decomposition~\eqref{lc} is that the first summand only involves the $\kappa_c$ which all have smaller degree than $2m$, namely $2m-c$. Therefore, we prefer decomposition~\eqref{lc} over the usual decomposition and work with and estimate the quantities $\kappa_c$ rather than Hoeffding's $\zeta_c$ which all have degree $2m$ (see also Remark~\ref{whyKappa} below).
\end{remark}
\subsection{Definition of the $U$-statistic for the variance}\label{UForVar}
In order to estimate the variance of the $U$-statistic $U$ by another $U$-statistic, we need the following.
\begin{assumption}
\label{assumptionA}
In the general situation of Proposition~\ref{VarIsU}, the statistic $U$ is non-degenerate. In the particular case $U =\U$ where $\Theta = \oParam$, this states that $\kappa_c\ne \oParam^2$ for $1\le c\le g + 1$.

Furthermore, we assume in the situation of Proposition~\ref{VarIsU} that all upper bounds for the degrees thus obtained are optimal. In the particular case $U =\U$, this means that the degree of $\kappa_c$ is $2m-c = 2g + 2 - c$ and that of $\Theta^2 = \oParam^2$ is $2m = 2g + 2$.
\end{assumption}
The non-degeneracy can be numerically checked for plausability, unlike Assumption~\ref{assumptionZ} and the degree optimality which both state that the regular parameters cannot be written by a smaller number of integrals. There seems to be no reason why a kernel of the form~\eqref{PhiformulaSymmetrized} with a non-trivial classifier should not satisfy them. The first part of Assumption~\ref{assumptionA} is needed for the central limit theorem~\ref{AA}, the second one for Theorem~\ref{MainTheorem}.

Proposition~\ref{VarIsU} motivates the following definition.
\begin{definition}\label{KappaDefi}
In the general situation of Proposition~\ref{VarIsU} and under Assumption~\ref{assumptionA}, the statistics $\widehat{\kappa_c}$ for $1\le c\le m$ of degree $2m - c$ and the statistic $\widehat{\Theta^2}$ of degree $2m$ are defined as the $U$-statistics associated to the symmetrized versions of the kernels which are the integrands in~\eqref{KappaIsRegPa} and~\eqref{ThetaSqIsRegPa}, respectively.
\end{definition}
Estimating $\Theta^2$ by $U^2$ instead would be biased and thus would not fit in our framework.
\begin{remark}\label{EstimateThetaSquaredAsWell}
 It would not be suitable to simply estimate $\Theta^2$ by zero in view of $H^0:\Theta = 0$. The first reason is that failure to subtract $(1 - \alpha_0)\cdot\widehat{\Theta^2}$ from the variance estimator~\eqref{VarEstForAnyU} below would overestimate the variance $\V{}{U}$ under $H^1$, leading to a severe loss of power. The second is that it would conflict with Hoeffding's setup. In fact, under the null hypothesis $\Theta = 0$, the degree of $\Theta$ and that of $\Theta^2$ would be trivially zero, if we were willing to restrict Hoeffding's class $\mathcal{P}$ of distributions to only ones obeying the null; however, the least-variance optimality property of a $U$-statistic relies on $\mathcal{P}$ encompassing \emph{all} properly discontinuous distribution functions, not only null ones. Likewise, the degree has to be defined for a global class of null and alternative together. This is akin of a classical one-way analysis of variance statistic where estimating variance within and between groups separately greatly increases the power.
\end{remark}
We can now define the variance estimator of a $U$-statistic as a $U$-statistic itself.
\begin{definition}\label{MainDefi}
In the general situation and notation of Proposition~\ref{VarIsU} and under Assumption~\ref{assumptionA}, we define an estimator, abbreviated $\wVU$, for the variance of the $U$-statistic $U$ as the $U$-statistic associated to the linear combination as in~\eqref{lc} of the kernels of $\kappa_c$ and of $\Theta^2$ given by~\eqref{KappaIsRegPa} and~\eqref{ThetaSqIsRegPa}.

After a short and straightforward computation, the definition can be re-stated alternatively in more explicit form: The single $U$-statistic $\wVU$ splits as a sum 
\begin{equation}\label{VarEstForAnyU}
\wVU \defequal \sum_{c = 1}^m \alpha_c\widehat{\kappa_c} - (1 - \alpha_0)\widehat{\Theta^2}
\end{equation}
of $U$-statistics of varying degrees. In the particular case $U =\U$ where $\Theta = \oParam$, this defines an estimator for $\V{}{\U}$, which will be abbreviated by $\VarEst$.
\end{definition}
The estimator $\wVU$ enjoys the unbiasedness and optimality properties analogous to $\U$. In particular, this applies to $\VarEst$.
\begin{remark}\label{whyKappa}
In the latter case $U = \U$, we have $m = g + 1$ so the degree of $\VarEst$ is $2g + 2$, and that of $\widehat{\kappa_c}$ was given in the degree optimality statement of Assumption~\ref{assumptionA}. The reason for splitting $\wVU$ into several $U$-statistics of varying degree is numerical efficiency: \citet{Hoeffding:1948} suggests to estimate $\zeta_c = \kappa_c - \Theta^2$. However, all of these parameters have degree $2m$. Instead, it is of course advisable to estimate the $\kappa_c$ which have smaller degee $2m - c$. Then, $\Theta^2$ needs to be estimated only once.
This remark is the empirical analogue to Remark~\ref{kioz}.

\end{remark}
\subsection{Existence criterion and order of consistency}
We are now in a position to investigate the estimator of $\V{}{\U}$. In principle, this section applies to the general situation of Proposition~\ref{UForVar}, but in order to keep the presentation clear we will focus on the interesting case $\Theta = \oParam, U = \U$ for the rest of the paper. Therefore, we will write $\widehat{\oParam^2}$ for the statistic $\widehat{\Theta^2}$ whereas we will not introduce a special notation for the statistics $\kappa_c$ for that case.

In the consistency statement of Theorem~\ref{MainTheorem} below, the true parameter $\V{}{\U}$ depends on $n$, unlike in a typical consistency statement. In principle, the sample size used for this estimation does not need to be the same $n$ again, but can in fact be any number $n' \ge 2g + 2$. However, in practise the same sample is used to estimate $\oParam$ as well as $\V{}{\U}$, so we restrict our attention to the diagonal case $n = n'$ for simplicity. This is analogous to the ordinary one-sample $t$-test statistic, where both the numerator, the sample mean, and its standard deviation, the denominator, are simultaneously estimated on the same sample, so with the same $n$. However, in our case, no factor $n^{-1/2}$ cancels out between both.
\begin{theorem}
\label{MainTheorem}
If $n\ge 2g + 2$, the estimator $\VarEst$ of $\V{}{\U}$ has least variance among all unbiased estimators of $\V{}{\U}$ over any family of distributions $\mathcal{P}$ containing all purely discontinuous distribution functions. Furthermore, $\VarEst$ is strongly consistent in the sense that $n^{d/2}(\VarEst - \V{}{\U}) \to 0$ almost surely for any $0 \le d \le 2$.
\end{theorem}
We do not claim to have exhibited the optimal order of consistency.
\begin{proof}[(of Theorem~\ref{MainTheorem})]
The unbiasedness of $\VarEst$ as well as its least-variance optimality follow from the general properties of $U$-statistics. Only the consistency statement remains to be shown. For $0 \le d < 1$, \citet[Equation 5.7]{Hoeffding:1963} applied to the $U$-statistic $\VarEst$ whose kernel is bounded between $0$ and $1$, yields the quantitative version
\begin{equation}\label{Bernstein1}
P\left(\abs{\VarEst - \V{}{\U}} \ge \epsilon n^{-d/2}\right) \le 2 \exp \left(-2 \floor{n/(2g+2)} \epsilon^2  n^{-d}\right)
\end{equation}
for any $\epsilon > 0$ which has to be applied with care because the degree of the $U$-statistic varies with $n$. This only applies to $0\le d \le 1$ and only shows weak consistency. In the following, we make use of the fact that $U$-statistics are strongly consistent meaning that they satisfy the strong law of large numbers if the kernel is absolutely integrable, for instance bounded. This was first shown in an unpublished paper of 1961 by Hoeffding, a complete proof is given in~\citet[Section 3.4.2, Theorem 3]{Lee:1990}. For all cases  $0 \le d \le 2$, let us first show that $n\VarEst$ almost surely tends to $(g + 1)^2(\kappa_1 - \oParam^2)$. For $c\ge 2$, the summands $n\alpha_c\widehat{\kappa}_c$ of $n\VarEst$ almost surely tend to zero, because $n\alpha_c$ does, and $\widehat{\kappa_c}$ is strongly consistent, so the sequence $\widehat{\kappa_c}$ for $n\to \infty$ is almost surely bounded, for every $c$.  For $c = 1$, the summand $n\alpha_1\widehat{\kappa_1}$ almost surely tends to $(g+1)^2\kappa_1$, because $n\alpha_1 \to (g + 1)^2$, and $\widehat{\kappa_1}$ is strongly consistent. Similarly, the summand $n(1 - \alpha_0)\widehat{\oParam^2}$ almost surely tends to $(g+1)^2\oParam^2$, because $n(1 - \alpha_0) \to (g + 1)^2$, and $\widehat{\oParam^ 2}$ is strongly consistent.

The statement now follows from the fact that $\lim_{n\to\infty} n\V{}{\U} = (g+1)^2(\kappa_1 - \oParam^2)$ \citep[5.23]{Hoeffding:1948}.
\end{proof}
Under $H^0$, there are unbiased estimators already for smaller $m$ since then $\oParam^2$ does not need to be estimated. However, as noted in Remark~\ref{EstimateThetaSquaredAsWell}, the optimality property cannot be shown in this case.
\section{Testing}\label{Testing}
\subsection{Central limit theorem}
The convergence of $\U$ towards $\oParam$ as $n\to \infty$ is described by the Strong Law of Large Number, the Law of the Iterated Logarithm and the Berry-Esseen theorem \citep[Section 3.4.2, Theorem 3, Section 3.5, Theorem 1, Section 3.3.2, Theorem 1, respectively]{Lee:1990}. In order to show the existence of an asymptotically exact test, we need the following theorem as it subsumes the unstudentized and the studentized case. It is reminiscent of and contains as special case the statement that the $t$-distributions tend to $\mathcal{N}(0, 1)$ as the degrees of freedom tend to infinity.
\begin{theorem}\label{AA}
Let $u(n)$ be one of the following expressions: the asymptotic variance $u(n) \defequal (g+1)^2\bigl(\kappa_1 - \oParam^2\bigr)/n$, the expression $u(n) \defequal (g+1)^2\bigl(\widehat{\kappa_1} - \widehat{\oParam^2}\bigr) / n$, where $\widehat{\kappa_1}$ and $\widehat{\oParam^2}$ are defined by the case $\Theta = \oParam, U = \U$ in Definition~\ref{KappaDefi}, or $u(n) \defequal \VarEst$ as of Definition~\ref{MainDefi}.

Then $(\widehat{\oParam} - \oParam)u(n)^{-1/2}$ converges in distribution to $\mathcal{N}(0, 1)$ as $g$ remains fixed, $n\to\infty$.
\end{theorem}
The occurrence of the factor $(g+1)^2$ is explained by the fact that this is the decay rate of the coefficient $\alpha_1$ in the sense that $\lim_{n\to\infty} n\alpha_1 = (g + 1)^2$. This also explains the asymptotic behaviour of the variance.

The first case of Theorem~\ref{AA} shows approximate normality of the unstudentized statistic $\oParam$ itself. It seems that there exists no statement in the literature giving the precise reason why a cross-validation type estimator is asymptotically normally distributed. This case appears in the asymptotic variance statement \citet[5.23]{Hoeffding:1948}. The second case is included for systematic reasons; this expression is the empirical analogue of the first case, but is a biased variance estimator. Finally, the third case includes the unbiased variance estimator elaborated in the present manuscript. The fact that $\V{}{\U} / u(n)$ tends to one, shown in the following proof, is not immediate, due to the diagonality property $n=n'$ mentioned above. Likewise, it is not obvious whether this ratio almost surely tends to one.
\begin{proof}[(of Theorem~\ref{AA})]
In the first case, this is \citet[Theorem 7.1]{Hoeffding:1948} and rests on the validity of the first part of Assumption~\ref{assumptionA}. In the other cases, the proof proceeds simultaneously. First, the proof of Theorem~\ref{MainTheorem} shows that convergence of $n\V{}{\U}$ implies the almost sure convergence not only of $\VarEst$ but in fact of $n u(n)$ for any of the choices of $u(n)$. Thus, $(n u(n))^{-1}$ is almost surely bounded. This statement is licit because the first part of Assumption~\ref{assumptionA} implies $\kappa_1\ne\oParam^2$, so $nu(n) $ converges to a non-zero value and, therefore, there are at most only finitely many $n$ such that $u(n) = 0$ has positive probability. Consequently, we may multiply $n(u(n) - \V{}{\U})$ which converges almost surely to zero, hence also in probability, with $(n u(n))^{-1}$ to show that in each case, the ratio $\V{}{\U} / u(n)$ tends to one in probability by Slutsky's theorem. By the continuous mapping theorem, $(\V{}{\U} / u(n))^{1/2}$ tends to one in probability as well. Therefore,
$
(\widehat{\oParam} - \oParam)u(n)^{-1/2} = (\widehat{\oParam} - \oParam)(\V{}{\U})^{-1/2}  (\V{}{\U} / u(n))^{1/2}
$
tends to $\mathcal{N}(0, 1)$ in distribution by the first case and another application of Slutsky's theorem.
\end{proof}
\subsection{Asymptotic rejection regions and confidence intervals}
So, the two-sided test of $H^0$ with the rejection region
\begin{equation}\label{test}
\left\{
\abs{\U}\ge u(n)^{1/2} \phi^{-1}(1 - \alpha/2)
\right\}
\end{equation}
has asymptotic level $\alpha$, where $\phi$ is the standard normal cumulative distribution function. While the second case of Theorem~\ref{AA} uses a positively biased variance estimator and hence provides a conservative test which, however, is asymptotically exact, the third case provides the best approximation to exactness already in the finite case. Likewise, an asymptotically exact confidence interval for $\Delta e$ at level $1 - \alpha$ is
\begin{equation}\label{confint}
\Big[
\U - u(n)^{1/2} \phi^{-1}(1 - \alpha/2), \U + u(n)^{1/2} \phi^{-1}(1 - \alpha/2)
\Big].
\end{equation}

A related, but different approach to a similar testing problem is provided by the so-called empirical Bernstein inequalities in \citet[Equations 12, 13]{Peel:2010}. These are sharp empirical inequalities for general $U$-statistics associated to bounded kernels. However, $n$ has to be an integer multiple of the degree, and the authors do not consider cross-validation, but only partitions of the test set.
\section{The convergence of incomplete to complete $U$-statistics in practise}\label{Practice}
In practical applications, the number of summands of~\eqref{Uformula} is too large for computation. In the particular case where one of the learning methods $M$ is a $k$-nearest-neighbour algorithm, it is possible to compute the corresponding summand of the complete $U$-statistic, the leave-$p$-out cross-validation estimator of the error rate, by an efficient closed-form expression~\citep{Celisse:2012}. In general, however,  one can only consider a design $\mathscr{T}$ smaller than the full $\mathscr{T}_{\text{\emph{max}}}$, leading to incomplete $U$-statistics as treated in \citet{Lee:1990}, for instance. We now show that the incomplete $U$-statistic with random design approximates the complete one satisfactorily after a feasible number of iterations.

Let $\Phi$ be a not necessarily symmetric kernel with $-1 \le \Phi \le 1$, let $\mathscr{T}^\ast$ be a collection of $N$ randomly drawn ordered size $m$-subsets of $\{1, \dotsc, n\}$ from the equidistribution $Q$ on the collection of such subsets, and let $\Phi(\mathscr{T}^\ast)$ be the associated incomplete $U$-statistic. Then the probability of approximation error at least $\delta> 0$ is bounded by
\begin{equation}\label{Fast}
\operatorname{\it pr}_Q(\abs{\Phi(\mathscr{T}^\ast) - \Phi(\mathscr{T})} \ge \delta) \le 2 \exp\bigl(-\delta^2N/2 \bigr).
\end{equation}
This follows from \citet[Theorem 2]{Hoeffding:1963} because the entries of $\mathscr{T}^\ast$ were drawn independently from each other. One should be aware that here we do \emph{not} refer to the part of \citet{Hoeffding:1963} concerned with $U$-statistics, in contrast to the situation of the similar inequality~\eqref{Bernstein1}, where we did so. Here, we formulated the version for ordered subsets because this is of immediate interest for computation.

The fast exponential decay of~\eqref{Fast} implies that sufficiency of the approximation is assured as soon as $N$ is a small multiple of $\delta^{-2}$, where $\delta$ is the pre-specified tolerance. Precisely, the following corollary of Hoeffding's theorem can be used in practise.
\begin{corollary}\label{fast}
After at most $N = 2d + 1$ iterations, $d$ digits after the comma are fixed with a probability of at least $1 - 2\exp(-5)\approx 0.99$.
\end{corollary}
Such a number of repetitions is, in general, hard but feasible because this $N$ is the mere number of times a model has to be fitted. For instance, in the illustration in Section~\ref{Results}, no tuning of the hyper-parameter $\lambda$ is part of each iteration. Remarkably, this bound on $N$ holds true irrespectively of the sample size $n$ or of any properties of the particular $U$-statistic under consideration, apart from $-1 \le \Phi \le 1$. In practise, however, one proceeds again slightly differently. For the case of approximation of the $U$-statistic $\U$ for instance, one applies the following procedure which yields even faster convergence against the true $\U$. In the formal setting required for inequality~\eqref{Fast}, one would use only one test observation for each learning iteration, which would lead to unnecessarily high computational cost. Instead, one simply uses all remaining $n-g$ observations for testing. This speeds up convergence even further. Corollary~\ref{fast} also applies to the computation of $\VarEst$ by the linear combination of $U$-statistics explained above because the kernels appearing in~\eqref{KappaIsRegPa} and~\eqref{ThetaSqIsRegPa} are bounded between $-1$ and $1$ as well.
\section{The calculations in a real data example}\label{Results}
The estimation procedure elaborated in the preceding sections was applied to the well-investigated colon cancer data set by~\citet{Alon:1999}, where the binary response $y\in\mathcal{Y}$ stands for the type of tissue (either normal tissue or tumor tissue) and the 2000 continuous predictors are gene expressions.
We used lasso-penalized logistic regression with the coordinate descent method for classification~\citep{Friedman:2010} and the  penalization parameters $\lambda =$ 0$\cdot$08, 0$\cdot$5. Pre-selecting these values led to the software-internal estimator for the difference of error rates to be greater than 0$\cdot$1. This involved the whole data set, however, this is no problem here. Sample size was $n = 62$. Therefore, the condition $n \ge 2g + 2$ constrained $g \le 30$. Since the variance of the $U$-statistic $\VarEst$ decreases to the extent to which the sample size exceeds the degree $2g + 2$, the learning set size $g$ was arbitrarily chosen to be only 26 to compromise with the effort to avoid a too small learning set size.

There were numerical evidence for the validity of the non-degeneracy statement of Assumption~\ref{assumptionA}. The resulting point estimate of  $\oParam$ was $-0\cdot 14$, with 95\%-confidence interval [-0$\cdot$35, 0$\cdot$07] and estimated variance $\VarEst =$ 0$\cdot$01. The number of iterations was $N=10^5$ for each of the $U$-statistics $\widehat{\oParam}, \widehat{\kappa_c}$ for $1\le c \le g+1$ and $\widehat{\oParam^2}$. By Corollary~\ref{fast}, two digits of each of these were therefore assured. The two-sided $p$-value was $p=$0$\cdot$19, given by the corresponding upper and lower normal tail probabilities. An R-script that allows to reproduce these results is available on the first author's institution web page.



\section*{Acknowledgement}
MF was supported by the German Science Foundation (DFG-Einzelf\"orderung BO3139/2-2 to ALB). 
RH was supported by the German Science Foundation (DFG-Einzelf\"orderung BO3139/3-1 to ALB). RDB was supported by the German Science Foundation (DFG-Einzelf\"orderung BO3139/4-1 to ALB).

\end{document}